\documentclass[graybox]{svmult}

\usepackage{mathptmx}       
\usepackage{helvet}         
\usepackage{courier}        
\usepackage{type1cm}        
\usepackage{makeidx}         
\usepackage{graphicx}        
\usepackage{multicol}        
\usepackage[bottom]{footmisc}

\makeindex             

\usepackage[utf8]{inputenc}
\usepackage{amsmath}
\usepackage{amsfonts}
\usepackage{amssymb}
\newtheorem{cor}{Corollary}[section]
\newtheorem{prop}{Proposition}[section]
\def\bhag#1{\noindent
\section{#1}
}

\def\seqb{{\mathsf{b}}}

\def\RR{\mathbb{R}}

\def\ZZ{{\mathbb Z}}
\def\PP{{\mathbb{P}}}
\def\PPI{{{\rm I}\kern-1pt\Pi}}

\def\a{\alpha}
\def\b #1;{{\bf #1}}

\def\O{{\cal O}}

\def\C{{\mathcal{C}}}

\def\esssup{\mathop{\hbox{{\rm ess sup}}}}

\def\be{\begin{equation}}
\def\ee{\end{equation}}
\def\bea{\begin{eqnarray}}
\def\eea{\end{eqnarray}}
\def\eref#1{(\ref{#1})}
\def\disp{\displaystyle}

\def\donchitre#1#2{\vskip 6.5cm\noindent
\parbox[t]{1in}{\special{eps:#1.eps x=6.5cm y=5.5cm}}
\hbox to 7cm{}\parbox[t]{0.0cm}{\special{eps:#2.eps x=6.5cm y=5.5cm}}}

\def\tn{|\!|\!|}

\def\BB{{\mathbb{B}}}
\def\span{\mathsf{span }}

\def\seqnorm#1{[\![{#1}]\!]}
\begin{document}

\title*{Local approximation using Hermite functions}
\author{H. N. Mhaskar}
\institute{H. N. Mhaskar \at Department of Mathematics, California Institute of Technology, Pasadena, CA 91125, USA and
            Institute of Mathematical Sciences, Claremont Graduate University, Claremont, CA 91711, USA , \email{hrushikesh.mhaskar@cgu.edu}
}
%
%
\maketitle

\abstract*{
We develop a wavelet like representation of functions in $L^p(\mathbb{R})$ based on their Fourier--Hermite coefficients; i.e., we describe an expansion of such functions where the local behavior of the terms characterize completely the local smoothness of the target function. In the case of continuous functions, a similar expansion is given based on the values of the functions at arbitrary points on the real line. In the process, we give new proofs for the localization of certain kernels, as well as some very classical estimates such as the Markov--Bernstein inequality.
}

\abstract{We develop a wavelet like representation of functions in $L^p(\mathbb{R})$ based on their Fourier--Hermite coefficients; i.e., we describe an expansion of such functions where the local behavior of the terms characterize completely the local smoothness of the target function. In the case of continuous functions, a similar expansion is given based on the values of the functions at arbitrary points on the real line. In the process, we give new proofs for the localization of certain kernels, as well as for some very classical estimates such as the Markov--Bernstein inequality.}

\bhag{Introduction}\label{intsect}

The subject of weighted polynomial approximation is by now  fairly 
well studied in approximation theory, with several books (e.g., \cite{mhasbk, safftotikbk, levinlubinsky}) devoted to various aspects of this subject. One of the first papers in the modern theory was by Freud, Giroux, and Rahman \cite{freud_giroux_rahman}. The purpose of this paper is to revisit this theory in the context of approximation by Hermite functions.

To describe our motivation, we consider the case of uniform approximation of periodic functions by trigonometric polynomials. In view of the direct and converse theorems of approximation, both the functions
$$
f_1(x)=\sqrt{|\cos x|}, \qquad f_2(x)=\sum_{k=0}^\infty \frac{\cos(4^kx)}{2^k}, \qquad x\in\RR,
$$
are in the same H\"older class $\mbox{Lip}(1/2)$, with the uniform degree of approximation by trigonometric polynomials of order $<n$ to both of these being $\O(n^{-1/2})$. However, $f_1$ has an analytic extension except at  $x=(2k+1)\pi/2$, $k\in\ZZ$, while $f_2$ is nowhere differentiable. Also, the Fourier coefficients of neither of the two functions reveal this fact. One of the reasons for developing the very popular wavelet analysis is to be able to detect the fact that $f_1$ is  only locally in $\mbox{Lip}(1/2)$ at $x=(2k+1)\pi/2$, $k\in\ZZ$, and infinitely smooth at other points by means of local behavior of the wavelet coefficients of $f_1$ rather than its Fourier coefficients \cite[Chapter~9]{daubbook}. 
Motivated by this theory, we have developed in a series of papers (e.g., \cite{trigwave, loctrigwave, locjacobi, jaenpap, locsmooth, fasttour, expframe, frankbern, heatkernframe, compbio, chuiinterp, treepap}) a theory of wavelet--like representations of functions on the torus, compact interval, sphere, manifolds, and graphs using the expansion coefficients of classical orthogonal systems on these domains, for example, Jacobi polynomials on the interval. 
In this paper, we develop  such a theory for the whole real line using Hermite functions as the underlying orthogonal system. 

Naturally, the basic ideas and ingredients involved this development are the same as in our previous work. However, there are several technical difficulties.  The infinite--finite range inequalities (see 
 Proposition~\ref{spuriousprop}) help us, as expected, to deal with the fact that the domain of approximation here is obviously not compact. 
An additional technical difficulty is the following ``product problem''. The product of two polynomials $P_1$, $P_2$ of degree $<n$ is also a polynomial of degree $<2n$. In contrast, the product of two ``weighted polynomials'' $\exp(-x^2/2)P_1(x)$, $\exp(-x^2/2)P_2(x)$ is not another weighted polynomial. A straightforward attempt to approximate $\exp(-x^2/2)$ by its Taylor polynomial or even the more sophisticated approach described in \cite[Chapter~7]{mhasbk} are not adequate to obtain the correct rates of approximation of such a product with weighted polynomials. 
The other important components in our theory are the availability of localized kernels and quadrature formulas based on arbitrary points on $\RR$. 
While the localization estimates on certain kernels as in Theorem~\ref{hermite_loc_theo} are given in \cite{dziubanski1997triebel, epperson1997hermite}, 
we give a more elementary proof based on the Mehler identity and a new Tauberian theorem proved in \cite{tauberian}. As a consequence, we also give a new proof of certain
classical inequalities such as the estimates on the Christoffel functions and Markov--Bernstein inequalities. 

The paper is organized as follows. We define the basic notations and definitions and summarize some preliminary facts in Section~\ref{notationsect}. In Section~\ref{wtapporoxsect}, we develop the machinery to help us surmount the product problem by reviewing and interpreting certain equivalence theorems from the theory of weighted polynomial approximation.
 Localized kernels will be described next in Section~\ref{summkernsect} (Theorem~\ref{hermite_loc_theo}). These will be used in Section~\ref{summopsect} to develop certain localized, uniformly bounded summability operators (Lemma~\ref{sigmabdlemma}, Theorem~\ref{goodapproxtheo}). In turn, these will be used to give a new proof of the Markov--Bernstein inequality in Corollary~\ref{bernsteincor}. The summability operators are analogues of the shifted average operators in \cite[Section~3.4]{mhasbk}. 
When defined in terms of the Lebesgue measure, they reproduce weighted polynomials. This may not hold when they are defined with other measures. For this purpose, we will prove in Section~\ref{quadsect} the existence of measures supported on an aribitrary set of real numbers which integrate products of weighted polynomials exactly.  Finally, the wavelet--like representation is given in  Section~\ref{mainsect}.

\bhag{Basic notation and definitions}\label{notationsect}
In this section, we collect together different notations and definitions, as well as some preliminary facts which we will use often in this paper.

If $x\in \RR$ and $r\ge 0$, we will write $\BB(x,r)=[x-r,x+r]$.

 Let $\{\psi_j\}$ denote the sequence of orthonormalized Hermite functions; i.e., \cite[Formulas~(5.5.3), (5.5.1)]{szego}
\be\label{hermitedef}
\psi_j(x)= \frac{(-1)^j}{\pi^{1/4}2^{j/2}\sqrt{j!}}\exp(x^2/2)\left(\frac{d}{dx}\right)^j (\exp(-x^2)), \qquad x\in\RR, \ j=0,1,\cdots.
\ee
We note that
\be\label{hermiteortho}
\int_\RR \psi_j(z)\psi_\ell(z)dz =\delta_{j,\ell},\qquad j,\ell=0,1,\cdots.
\ee
We denote $w(x)=\exp(-x^2/2)$. For $t>0$, let $\PP_t$ be the class of all algebraic polynomials of degree $<t$. The space $\Pi_t$ is defined by
\be\label{approxspacedef}
\Pi_t=\span\{\psi_j : \sqrt{j}<t\}=\{wP : P\in\PP_{t^2}\}, \qquad t>0.
\ee

In this paper, the term measure will denote a signed, complex valued Borel measure (or a positive, sigma--finite Borel measure). We recall that if $\mu$ is an extended complex valued Borel measure on $\RR$, then its total variation measure is defined for a Borel set $B$ by
$$
|\mu|(B)=\sup\sum |\mu(B_k)|,
$$
where the sum is over a partition $\{B_k\}$ of $B$ comprising Borel sets, and the supremum is over all such partitions.

\begin{definition}\label{regmeasdef}
 If $t>0$, a Borel measure $\nu$ will be called \textbf{$t$--regular} if there exists a constant $A>0$ such that
\be\label{regularmeasdef}
|\nu|(\BB(x,r))\le A(r+1/t), \qquad x\in\RR, \ r>0.
\ee
We will define the regularity norm of $\nu$ by
\be\label{regularmeasnorm}
\tn\nu\tn_{t}=\sup_{r>0}\frac{|\nu|(\BB(x,r))}{r+1/t}.
\ee
The set of all Borel measures for which $\tn\nu\tn_{t}<\infty$ is a vector space, denoted by $\mathcal{R}_{t}$. \qed
\end{definition}
It is easy to verify that  $\tn \cdot\tn_{t}$ is a norm on $\mathcal{R}_t$. 
It is not difficult to deduce from the definition that
$$
\tn\nu\tn_t\le \max(1, t/u)\tn\nu\tn_u, \qquad t, u>0.
$$
In particular, when  $t<u$, $\mathcal{R}_u\subseteq \mathcal{R}_t$, and for any constant $c>0$, the spaces of measures $\mathcal{R}_t$ and $\mathcal{R}_{ct}$ are the same, with the constants involved in the norm equivalence depending upon $c$.
 
For example, the Lebesgue measure on $\RR$ is in $\mathcal{R}_\infty$, and its regularity norm is obviously $1$. If $\mathcal{C}\subset\RR$,  the \textbf{density content} of $\mathcal{C}$ is defined by
\be\label{meshnormdef}
\delta(\mathcal{C})=\sup_{y,z\in\mathcal{C}}|y-z|.
\ee
If $\mathcal{C}$ is a finite set, and $\nu$ is a measure that associates the mass $1$ with each of these points then $\nu$ is clearly $1/\delta(\mathcal{C})$--regular.
\begin{definition}\label{quadmzmeasdef}
Let $n>0$. A Borel measure $\nu$ on $\RR$ is called \textbf{quadrature measure} of order $n$ if
\be\label{genquad}
\int_\RR P(y)Q(y)dy=\int_\RR P(y)Q(y)d\nu(y), \qquad P, Q\in\Pi_n.
\ee
The set of all  quadrature measures  of order $n$ which are in $\mathcal{R}(n)$ is denoted by $MZ(n)$. \qed
\end{definition}

We note that the formula \eref{genquad} is required for \textbf{products of weighted polynomials}.
Clearly, the Lebesgue measure itself is in $MZ(n)$ for all $n>0$. In Theorem~\ref{quadtheo}, we will prove the existence of  measures in $MZ(n)$ supported on a sufficiently dense set of points in $\RR$.

If $\nu$ is any Borel measure on $\RR$, for $1\le p\le \infty$, and $\nu$--measurable set $B\subseteq \RR$ and $\nu$--measurable function $f:B\to\RR$
$$
\|f\|_{\nu;p,B}:=\left\{\begin{array}{ll}
\disp\left\{\int_B |f(x)|^pd|\nu|(x)\right\}^{1/p}, &\mbox{ if $1\le p<\infty$,}\\
|\nu|-\esssup_{x\in B}|f(x)|, &\mbox{ if $p=\infty$.}
\end{array}\right.
$$
The class of all functions $f$ for which $\|f\|_{\nu;p,B}<\infty$ is denoted by $L^p(\nu;B)$, with the usual convention that functions that are equal $|\nu|$--almost everywhere are considered to be equal. If $\nu$ is the Lebesgue measure, its mention will be omitted from the notation, and if $B=\RR$, its mention will also be omitted from the notation. The set $X^p$ will denote $L^p$ if $1\le p<\infty$, and the set of all continuous functions on $\RR$ which vanish at infinity if $p=\infty$.

\paragraph{\textbf{Constant convention}}

The symbols $c, c_1,\cdots$ will denote generic positive constants depending only on the fixed parameters in the discussion, such as  the norms, smoothness parameters, etc. Their value may be different at different occurrences, even within a single formula. The notation $A\sim B$ means that $c_1A\le B\le c_2A$. \qed

\bhag{Weighted approximation}\label{wtapporoxsect}
In this section, we review some results from \cite{tenswt} for the sake of making this paper more self--contained. The main purpose is to point out Corollary~\ref{wtchangecor}, which will help us later in Section~\ref{mainsect} to get around the difficulty that the product of $P, Q\in \Pi_n$ is not in any $\Pi_{cn}$. 

Let $1\le p\le\infty$, $t>0$. If $f\in L^p$, we define
\be\label{degapproxdef}
E_{t,p}(f)=\inf_{Q\in\Pi_t}\|f-Q\|_p.
\ee

 For $t>0$ and integer $k\ge 0$, the forward difference of a function $f:\RR \to\RR$ is defined by
$$
\Delta_t^k f(x) := \sum_{\ell=0}^{k}(-1)^{k-\ell}\left(\begin{array}{c}
k\\ \ell\end{array}\right)f(x+\ell t).
$$
 With  
$$
Q_\delta (x):= \min \left(\delta^{-1}, (1+x^2)^{1/2}\right), \qquad \delta>0,\ x\in\RR,
$$ 
we define a modulus of smoothness for $f\in L^p$, $\delta>0$ by the formula 
\be\label{upremoddef}
\omega_r(p;f,\delta):= \sum_{k=0}^r \delta^{r-k}\sup_{|t|\le\delta}\|Q_\delta^{r-k}\Delta_t^k f\|_p.
\ee

The results in \cite{tenswt} lead to the following theorem.
\begin{theorem}\label{directconvtheo}
Let $1\le p\le\infty$, $f\in X^p$, $r, n\ge 1$ be  integers. Then
\be\label{directtheo}
E_{n,p}(f)\le c\omega_r(p;f,1/n),
\ee 
and
\be\label{convtheo}
\omega_r(p;f,1/n)\le \frac{c}{n^r}\left\{\|f\|_p+\sum_{k=0}^n (k+1)^{r-1}E_{k,p}(f)\right\}.
\ee
\end{theorem}

For the present paper, we need the following equivalence theorem Theorem~\ref{equivtheo} which is obtained from Theorem~\ref{directconvtheo} using standard methods of approximation theory as in \cite{devlorbk}.

For a sequence $\mathbf{a}=\{a_n\}_{n=0}^\infty$, $0<\rho\le \infty$, $\gamma\in (0,\infty)$, we define the sequence (quasi--)norm
\be\label{seqbesovnormdef}
\seqnorm{\mathbf{a}}_{\rho,\gamma}=\left\{\begin{array}{ll}
\left(\sum_{n=0}^\infty (2^{\gamma n}|a_n|)^\rho\right)^{1/\rho}, &\mbox{ if $0<\rho<\infty$,}\\
\sup_{n\ge 0} 2^{n\gamma}|a_n|, &\mbox{ if $\rho=\infty$.}
\end{array}\right.
\ee
The space of all sequences $\mathbf{a}$ with $\seqnorm{\mathbf{a}}_{\rho,\gamma}<\infty$ will be denoted by $\seqb_{\rho,\gamma}$.

\begin{definition}\label{besovspacedef}
Let $1\le p\le \infty$, $0<\rho\le\infty$, $0<\gamma<\infty$. The \textbf{Besov space} $B_{p,\rho,\gamma}$ is the space of all $f\in X^p$ for which $\|f\|_p+\seqnorm{\{E_{2^n,p}(f)\}_{n=0}^\infty} <\infty$. \qed
\end{definition}

\begin{theorem}\label{equivtheo}
Let $0<\rho\le\infty$, $0<\gamma<\infty$, $1\le p\le\infty$, $f\in X^p$, and 
$r>\gamma$ be an integer. Then $f\in B_{p,\rho,\gamma}$ if and only if
$\seqnorm{\{\omega_r(p;f, 1/2^n)\}_{n=0}^\infty} <\infty$.
\end{theorem}

A consequence of this theorem is the following. Let $w(x)=\exp(-x^2/2)$. Let $1\le p\le\infty$, $t>0$. If $f\in L^p$, we define
\be\label{tildedegapproxdef}
\tilde{E}_{t,p}(f)=\inf_{R\in\PP_{t^2}}\|f-Rw^2\|_p.
\ee
With $\tilde{f}(x)=f(x/\sqrt{2})$, it is elementary to see that $\tilde{E}_{n,p}(f)\sim E_{n,p}(\tilde{f})$. Since\\ $\omega_r(p;\tilde{f},\delta)\sim \omega_r(p;f,\delta)$ for $\delta>0$, we obtain as a corollary to Theorem~\ref{equivtheo} the following.

\begin{cor}\label{wtchangecor}
Let $0<\rho\le\infty$, $0<\gamma<\infty$, $1\le p\le\infty$, $f\in X^p$. Then $f\in B_{p,\rho,\gamma}$ if and only if
$\seqnorm{\{\tilde{E}_{2^n,p}(f)\}_{n=0}^\infty} <\infty$.
\end{cor}
\bhag{Localized  kernels}\label{summkernsect}

If $H :[0,\infty)\to \RR$ is a compactly supported function, we write
\be\label{hermite_kerndef}
\Phi_n(H;x,y)=\sum_{j=0}^\infty H\left(\frac{\sqrt{j}}{n}\right)\psi_j(x)\psi_j(y), \qquad n >0,\ x, y\in\RR.
\ee

\begin{theorem}\label{hermite_loc_theo}
Let $H :\RR\to \RR$ be a compactly supported, infinitely differentiable, even function. For $x, y\in\RR$, $n\ge 1$, $S\ge 3$, we have
\be\label{hermite_lockernest}
|\Phi_n(H;x,y)|\le c\frac{n}{\max(1, (n|x-y|)^S)}, \qquad \left|\frac{\partial}{\partial x}\Phi_n(H; x,y)\right|\le c\frac{n^2}{\max(1, (n|x-y|)^S)},
\ee
where the constants $c$ may depend upon $S$.
\end{theorem}

The proof of this theorem requires some preparation. First, we recall some terminology.

A measure $\mu$ on $\RR$ is called an even measure if $\mu((-u,u))=2\mu([0,u))$ for all $u>0$, and $\mu(\{0\})=0$. If $\mu$ is an extended complex valued measure on $[0,\infty)$, and $\mu(\{0\})=0$, we define a measure $\mu_e$ on $\RR$ by 
$$
\mu_e(B)=\mu\left(\{|x| : x\in B\}\right),
$$
and observe that $\mu_e$ is an even measure such that $\mu_e(B)=\mu(B)$ for $B\subset [0,\infty)$. In the sequel, we will assume that all measures on $[0,\infty)$ which do not associate a nonzero mass with the point $0$ are extended in this way, and will abuse the notation $\mu$ also to denote the measure $\mu_e$. In the sequel, the phrase ``measure on $\RR$'' will refer to an extended complex valued Borel measure having bounded total variation on compact intervals in $\RR$, and similarly for measures on $[0,\infty)$.

The proof of Theorem~\ref{hermite_loc_theo} uses two Tauberian theorems. 
The first of these \cite[Theorem~2.1]{tauberian} is the following.

\begin{theorem}\label{maintaubertheo}
Let $\mu$ be an extended complex valued measure on $[0,\infty)$, and $\mu(\{0\})=0$. We assume that there exist $Q, r>0$, such that each of the following conditions are satisfied.
\begin{enumerate}
\item 
\be\label{muchristbd}
\sup_{u\in [0,\infty)}\frac{|\mu|([0,u))}{(u+2)^Q} <\infty,
\ee
\item There are constants $c, C >0$,  such that
\be\label{muheatgaussbd}
\left|\int_\RR \exp(-u^2t)d\mu(u)\right|\le c_1t^{-C}\exp(-r^2/t)\sup_{u\in [0,\infty)}\frac{|\mu|([0,u))}{(u+2)^Q}, \qquad 0<t\le 1.
\ee 
\end{enumerate}
Let $H:[0,\infty)\to\RR$, $S>Q+1$ be an integer, and suppose that there exists a measure $H^{[S]}$ such that
\be\label{Hbvcondnew}
H(u)=\int_0^\infty (y^2-u^2)_+^{S}dH^{[S]}(y), \qquad u\in\RR,
\ee
and
\be\label{Hbvintbdnew}
V_{Q,S}(H)=\max\left(\int_0^\infty (y+2)^Qy^{2S}d|H^{[S]}|(y), \int_0^\infty (y+2)^Qy^Sd|H^{[S]}|(y)\right)<\infty.
\ee
Then for $n\ge 1$,
\be\label{genlockernest}
\left|\int_0^\infty H(u/n)d\mu(u)\right| \le c\frac{n^Q}{\max(1, (nr)^S)}V_{Q,S}(H)\sup_{u\in [0,\infty)}\frac{|\mu|([0,u))}{(u+2)^Q}.
\ee
\end{theorem}
The second theorem we need is the following \cite[Lemma~5.2]{eignet}.
 \begin{theorem}\label{tauberlemma}
Let $C>0$, $\{\ell_j\}$ be a non--increasing sequence of non--negative numbers such that $\ell_0=0$ and $\disp \lim_{j\to\infty}\ell_j=\infty$. Let $\{a_j\}$ be a sequence of nonnegative numbers such that $\sum_{j=0}^\infty \exp(-\ell_j^2t)a_j$ converges for $t\in (0,1]$. Then
\be\label{genchristlowbd}
c_1L^C\le \sum_{\ell_j\le L}a_j \le c_2L^C, \qquad L>0,
\ee
if and only if
\be\label{genheatlowbd}
c_3t^{-C/2}\le \sum_{j=0}^\infty \exp(-\ell_j^2t)a_j\le c_4t^{-C/2}, \qquad t\in (0,1].
\ee
\end{theorem}

We are now in a position to prove Theorem~\ref{hermite_loc_theo}. We note that the  estimates \eref{christbd} and \eref{diffchristbd} below were obtained in \cite[Theorem~3.3.4]{mhasbk} assuming the Markov--Bernstein inequality using more complicated machinery. In the present paper, the Markov--Bernstein inequality will be deduced as a consequence of Theorem~\ref{hermite_loc_theo}.

\noindent
\textit{Proof of Theorem~\ref{hermite_loc_theo}.} The starting point of the proof is
the Mehler formula \cite[Formula~(6.1.13)]{andrews_askey_roy}:
For $x,y\in\RR$,  $|r|<1$,
\bea\label{mehler}
\sum_{j=0}^\infty \psi_j(x)\psi_j(y)r^j&=& \frac{1}{\sqrt{\pi (1-r^2)}}\exp\left(\frac{2xyr-(x^2+y^2)r^2}{1-r^2}\right)\exp(-(x^2+y^2)/2)\nonumber\\
&=& \frac{1}{\sqrt{\pi (1-r^2)}}\exp\left(-\frac{r}{1-r^2}(x-y)^2-\frac{1-r}{1+r}\frac{x^2+y^2}{2}\right).
\eea
Writing $r=e^{-t}$, $t>0$, we get the explicit expression for the ``heat kernel'':
\bea\label{hermite_heatkern}
\lefteqn{\sum_{j=0}^\infty e^{-jt}\psi_j(x)\psi_j(y)}\nonumber\\
&=&\frac{e^{t/2}}{\sqrt{2\pi\sinh t}}\exp\left(-\frac{2}{\sinh t}(x-y)^2\right)\exp(-(1/2)\tanh (t/2)(x^2+y^2)).
\eea
Hence,
\be\label{hermite_gaussbd}
\left|\sum_{j=0}^\infty e^{-jt}\psi_j(x)\psi_j(y)\right|\le \frac{c_1}{\sqrt{t}}\exp\left(-\frac{c(x-y)^2}{t}\right), \qquad 0<t\le 1.
\ee
Taking $x=y$ above, we see that 
\be\label{pf1eqn1}
\sum_{j=0}^\infty e^{-jt}\psi_j(x)^2 \le ct^{-1/2}.
\ee
Consequently, Theorem~\ref{tauberlemma} used with $\ell_j=\sqrt{j}$ and $a_j=\psi_j(x)^2$ yields
\be\label{christbd}
\sum_{0\le \sqrt{j}<u} \psi_j(x)^2 \le cu, \qquad u\ge 1.
\ee
We now define a family of measures $\mu_{x,y}$ by
$$
\mu_{x,y}(u)=\sum_{0\le \sqrt{j}<u}\psi_j(x)\psi_j(y),  \qquad u, x, y\in\RR.
$$
Using Schwarz inequality and \eref{christbd}, we conclude that
\be\label{pf1eqn2}
\sup_{u>0}\frac{|\mu_{x,y}|(u)}{u+2}\le c, \qquad x, y\in\RR.
\ee
In view of \eref{hermite_gaussbd},  the estimate \eref{muheatgaussbd} is satisfied by  each  of the measures $\mu_{x,y}$ with $r=|x-y|$.  Moreover, it is clear that $H$ satisfies the conditions required in Theorem~\ref{maintaubertheo}.  Since
$$
\Phi_n(H; x,y)=\int_0^\infty H(u/n)d\mu_{x,y}(u), 
$$
we may use Theorem~\ref{maintaubertheo} with $Q=1$ to arrive at the first inequality in \eref{hermite_lockernest}.

In order to prove the second estimate in \eref{hermite_lockernest}, 
we define a family of measures  $\mu_{x,y}^{(1)}$ by
$$
 \mu_{x,y}^{(1)}(u)= \sum_{0\le \sqrt{j}<u}\psi_j'(x)\psi_j(y), \qquad u, x, y\in\RR,
$$
and observe that
$$
 \frac{\partial}{\partial x}\Phi_n(H;x,y)=\int_0^\infty H(u/n)d\mu_{x,y}^{(1)}(u), \qquad x, y\in \RR.
$$
We will  verify that \eref{muheatgaussbd} is satisfied by  each  of the measures $\mu_{x,y}^{(1)}$ with $r=|x-y|$, and
\be\label{pf1eqn3}
\sup_{u>0}\frac{|\mu_{x,y}^{(1)}|(u)}{(u+2)^2}\le c, \qquad x, y\in\RR.
\ee
An application of Theorem~\ref{maintaubertheo} with $Q=2$ then implies the desired second inequality in \eref{hermite_lockernest} as before.

Since $\psi_n'(x)=\sqrt{2n}\psi_{n-1}(x)-x\psi_n(x)$ (cf. \cite[Eqn.~(5.5.1), (5.5.10)]{szego}), it follows from \eref{christbd} that $\|\psi_n'\|_\infty \le cn^2$. Therefore, we may differentiate  the left hand side of \eref{hermite_heatkern} term by term to obtain for $t>0$
\bea\label{hermite_diffheatkern}
\lefteqn{\sum_{j=0}^\infty e^{-jt}\psi_j'(x)\psi_j(y)=\frac{e^{t/2}}{\sqrt{2\pi\sinh t}}\left\{\frac{4(y-x)}{\sinh t}-x\tanh (t/2)\right\}\times}\nonumber\\
&&\qquad\qquad
\exp\left(-\frac{2}{\sinh t}(x-y)^2-(1/2)\tanh (t/2)(x^2+y^2)\right), \eea
and
\bea\label{diffheatkerntwice}
\lefteqn{\sum_{j=0}^\infty e^{-jt}\psi_j'(x)\psi_j'(y)}\nonumber\\
&=&\frac{e^{t/2}}{\sqrt{2\pi\sinh t}}\left\{\frac{4}{\sinh t}+\left(\frac{4(y-x)}{\sinh t}-x\tanh (t/2)\right)\left(\frac{4(x-y)}{\sinh t}-y\tanh (t/2)\right)\right\}\times\nonumber\\
&&\qquad\exp\left(-\frac{2}{\sinh t}(x-y)^2-(1/2)\tanh (t/2)(x^2+y^2)\right).\nonumber\\
\eea
Since $\disp\max_{x\in\RR}|x|^m\exp(-ax^2)=(2a/(em))^{-m/2}$, $m=1,2,\cdots$, we deduce from \eref{hermite_diffheatkern} and \eref{diffheatkerntwice} that for $0<t\le 1$,
\be\label{hermite_diffgaussbd}
\left|\sum_{j=0}^\infty e^{-jt}\psi_j'(x)\psi_j(y)\right| \le \frac{c_1}{t}\exp\left(-\frac{c(x-y)^2}{t}\right), \quad \sum_{j=0}^\infty e^{-jt}\psi_j'(x)^2 \le ct^{-3/2}.
\ee
Thus,  each  of the measures $\mu_{x,y}^{(1)}$ satisfies \eref{muheatgaussbd} with $r=|x-y|$.
Using Theorem~\ref{tauberlemma}  with $\psi_j'(x)^2$ in place of $a_j$,  \eref{hermite_diffgaussbd} leads to
\be\label{diffchristbd}
 \sum_{0\le \sqrt{j}<u} \psi_j'(x)^2 \le cu^3, \qquad u\ge 1.
\ee 
Therefore, using Schwarz inequality and \eref{christbd}, we conclude that
 for $u\ge 1$,
$$
|\mu_{x,y}^{(1)}|(u)\le \sum_{0\le \sqrt{j}<u}|\psi_j'(x)\psi_j(y)| \le cu^2.
$$
This leads to \eref{pf1eqn3}, and completes the proof of the second inequality in \eref{hermite_lockernest} as explained before. \qed

\bhag{Summability operators}\label{summopsect}

\begin{definition}\label{lowpassfilterdef}
A function $h:\RR\to [0,1]$ is called a \textbf{low pass filter} if each of the following conditions is satisfied.
\begin{enumerate}
\item $h$ is an even, infinitely differentiable function on $\RR$,
\item $h(u)=1$ for $|u|\le 1/2$,
\item $h$ is non--increasing on $[1/2,1]$,
\item $h(u)=0$ if $|u|\ge 1$. \qed
\end{enumerate}
\end{definition}

In the sequel we will fix an infinitely differentiable low pass filter $h$, and will omit its mention from the notations, unless necessary to avoid confusion. In particular, the constants may depend upon $h$.

Let $n>0$, $\nu$ be a Borel measure on $\RR$, $f \in L^1(\nu)+L^\infty$, and $x\in\RR$. We define
\be\label{fourcoeffdef}
\hat{f}(\nu;j)=\int_\RR f(y)\psi_j(y)d\nu(y), \qquad j=0,1,\cdots,
\ee
and with $\Phi_n(x,y)=\Phi_n(h;x,y)$ as defined in \eref{hermite_kerndef},
\be\label{sigmaopdef}
\sigma_n(\nu;f,x)=\sigma_n(h;\nu;f,x)=\int_\RR \Phi_n(x,y)f(y)d\nu(y) =\sum_{j=0}^\infty h(\sqrt{j}/n)\hat{f}(\nu;j)\psi_j(x).
\ee
As usual, we will omit the mention of $\nu$ if $\nu$ is the Lebesgue measure on $\RR$, e.g.,
\be\label{fourcoeffdefbis}
\hat{f}(j)=\int_\RR f(y)\psi_j(y)dy, \qquad j=0,1,\cdots.
\ee
In this section, we will also find it useful to introduce the notation
\be\label{diffsigmaopdef}
\sigma_n^{(1)}(f,x)=\frac{d}{dx}\sigma_n(f,x), \qquad x\in\RR, \ f\in L^1+L^\infty.
\ee

The main theorem of this section is the following.

\begin{theorem}\label{goodapproxtheo}
Let $n>0$, $\nu\in MZ(n)$. If $P\in\Pi_{n/2}$ then $\sigma_n(\nu;P)=P$. If $1\le p\le \infty$ and $f\in L^p$, then
\be\label{goodapprox}
E_{n,p}(f)\le \|\sigma_n(\nu;f)-f\|_p \le cE_{n/2,p}(f).
\ee
\end{theorem}

In preparation for the proof of this theorem, we first prove two  lemmas.

\begin{lemma}\label{fundalemma}
If $t>0$,  $\nu\in\mathcal{R}_t$, $r>0$, $S\ge 2$, and $x\in\RR$, then
\be\label{fundaineq}
\int_{\RR\setminus\BB(x,r)}|y-x|^{-S}d|\nu|(y) \le \frac{2^S}{2^S-2}\tn\nu\tn_t r^{-S+1}(2+1/(rt)).
\ee
In particular, if $n>0$, and $\nu\in\mathcal{R}_n$ then 
\be\label{kernbds}
\int_\RR|\Phi_n(x,y)|d|\nu|(y)\le c\tn\nu\tn_n, \qquad \int_\RR \left|\frac{\partial}{\partial x}\Phi_n(x,y)\right|d|\nu|(y)\le cn\tn\nu\tn_n.
\ee
\end{lemma}

\begin{proof}\ 
By re--normalization if necessary, we may assume in this proof that $\tn\nu\tn_t=1$. Then \eref{regularmeasnorm} can be used to deduce that
\begin{eqnarray*}
\int_{\RR\setminus\BB(x,r)}|y-x|^{-S}d|\nu|(y)&=&\sum_{j=0}^\infty \int_{\BB(x,2^{j+1}r)\setminus \BB(x, 2^jr)} |y-x|^{-S}d|\nu|(y)\\
& \le& \sum_{j=0}^\infty (2^jr)^{-S}|\nu|(\BB(x,2^{j+1}r)) \\
&\le& \sum_{j=0}^\infty (2^jr)^{-S}(2^{j+1}r+1/t)=\frac{2^S r^{-S+1}}{2^{S-1}-1}+\frac{2^Sr^{-S}}{(2^S-1)t}\\
&\le& \frac{2^Sr^{-S+1}}{2^S-2}(2+1/(rt)).
\end{eqnarray*}
Using the first estimate in \eref{hermite_lockernest} with $S\ge 2$, we deduce from \eref{fundaineq} (with $n$ in place of $t$) that
\begin{eqnarray*}
\int_\RR|\Phi_n(x,y)|d|\nu|(y)&=& \int_{\BB(x,1/n)} |\Phi_n(x,y)|d|\nu|(y)+ \int_{\RR\setminus\BB(x,1/n)}|\Phi_n(x,y)|d|\nu|(y)\\
&\le& cn\left\{|\nu|(\BB(x,1/n))+n^{-S}n^{S-1}\right\}\le c.
\end{eqnarray*}
The second estimate in \eref{kernbds} is proved in the same way using the second estimate in \eref{hermite_lockernest}. \qed
 \end{proof}
 
 As a consequence of this lemma, we obtain the following.
 
\begin{lemma}\label{sigmabdlemma}
Let $n>0$, $\mu, \nu\in\mathcal{R}_n$, and $1\le p\le\infty$. Then
\be\label{sigmaopbds}
\|\sigma_n(\nu;f)\|_{\mu;p}\le c\|f\|_{\nu;p}, \qquad f\in L^p(\nu), 
\ee
\be\label{diffsigmaopbds}
\|\sigma_n^{(1)}(f)\|_p\le cn\|f\|_p, \qquad f\in L^p(\nu)
\ee
\end{lemma}

\begin{proof}\ 
In view of \eref{kernbds}, for all $x\in \RR$, and $f\in L^\infty(\nu)$,
$$
|\sigma_n(\nu;f,x)|\le \int_\RR |\Phi_n(x,y)||f(y)|d|\nu|(y)\le c\|f\|_{\nu;\infty},
$$
and similarly, using Tonnelli's theorem, if $f\in L^1(\nu)$,
\begin{eqnarray*}
\int_\RR |\sigma_n(\nu;f,x)|d|\mu|(x)&\le& \int_\RR\int_\RR |\Phi_n(x,y)||f(y)|d|\nu|(y)d|\mu|(x)\\
&=&\int_\RR\int_\RR |\Phi_n(y,x)||f(y)|d|\mu|(x)d|\nu|(y)\le c\|f\|_{\nu;1}.
\end{eqnarray*}
The estimate \eref{sigmaopbds} follows from these and the Riesz interpolation theorem. The proof of \eref{diffsigmaopbds} is similar. \qed
\end{proof}

We are now in a position to prove Theorem~\ref{goodapproxtheo}.

\textit{Proof of Theorem~\ref{goodapproxtheo}.}
We recall that $h(u)=1$ if $|u|\le 1/2$. If $P\in \Pi_{n/2}$, then for $x\in\RR$,
$$
P(x)=\sum_{0\le k <n/2}\hat{P}(j)\psi_j(x)=\sum_{k=0}^\infty h(\sqrt{j}/n)\hat{P}(j)\psi_j(x)=\sigma_n(P,x)=\int_\RR P(y)\Phi_n(x,y)dy.
$$
Since $\nu\in MZ(n)$, the definition \eref{genquad} now shows that
$$
P(x)=\int_\RR P(y)\Phi_n(x,y)d\nu(y)=\sigma_n(\nu;P,x).
$$

The first inequality in \eref{goodapprox} is obvious. In view of
 Lemma~\ref{sigmabdlemma}, we obtain for any $P\in\Pi_{n/2}$,
$$
\|\sigma_n(\nu;f)-f\|_p= \|\sigma_n(\nu;f-P)-(f-P)\|_p\le  c\|f-P\|_p .
$$
This leads to the second inequality in \eref{goodapprox}. \qed

We end this section by pointing out that the estimate \eref{diffsigmaopbds} leads immediately to the following Markov--Bernstein inequality. This deduction is the same in spirit as that given in \cite{mhasbk}, but we consider it to be a new proof, since the proof of \eref{diffsigmaopbds} is significantly different from that in \cite{mhasbk}.

\begin{cor}\label{bernsteincor}
 For $1\le p\le \infty$,
\be\label{bernineq}
\|P'\|_p\le cn\|P\|_p, \qquad n>0, \ P\in\Pi_n.
\ee
\end{cor}
\begin{proof}\ 
If $P\in \Pi_n$, Theorem~\ref{goodapproxtheo} shows that $\sigma_{2n}(P)=P$, so that $P'=\sigma_{2n}^{(1)}(P)$.
 The inequality \eref{bernineq} follows from this and \eref{diffsigmaopbds}. \qed
\end{proof}


\bhag{Quadrature formula}\label{quadsect}
In this section, we wish to demonstrate the existence of measures in $MZ(n)$, supported on sufficiently dense finite point sets in $\RR$, in the sense made precise below. We recall that if $\mathcal{C}\subset\RR$,  the density content of $\mathcal{C}$ is defined by
\be\label{meshnormdefbis}
\delta(\mathcal{C})=\sup_{y,z\in\mathcal{C}}|y-z|.
\ee

\begin{theorem}\label{quadtheo}
There exists $C, \a>0$ with the following property: With $A_n=(n\sqrt{2})(1+Cn^{-4/3})$,  if $\mathcal{C}=\{y_1<\cdots<y_{M+1}\}\subset \RR$,
 $[-A_n,A_n]\subseteq [y_1,y_{M+1}]$, and   $\delta(\C)\le c$, then there exist real numbers $w_1,\cdots, w_M$ such that with $n=\a\delta(\C)^{-1}$,
\be\label{quadrature}
\int_\RR P(y)Q(y)dy=\sum_{k=1}^M w_kP(y_k)Q(y_k), \qquad P, Q\in \Pi_n,
\ee
and
\be\label{quadwtbd}
|w_k|\le c|y_{k+1}-y_k|, \qquad k=1,\cdots,M.
\ee
In particular, the measure $\nu$ that associates the mass $w_k$ with each of the points $y_k$  is in $MZ(n)$. Further,
 if $[y_1,y_{M+1}]\subset [-cn^\beta, cn^\beta]$ for some $\beta>0$, then 
\be\label{bdvar}
\sum_{k=1}^M|w_k|\le cn^\beta.
\ee
\end{theorem}
 
 This theorem will be deduced by making some changes in variable in the following theorem.
 \begin{theorem}\label{singquadtheo}
There exists $C, \a_1>0$ with the following property: With $A'_n=2n(1+Cn^{-4/3})$,  if $\mathcal{C'}=\{x_1<\cdots<x_{M+1}\}\subset \RR$,
 $[-A'_n,A'_n]\subseteq [x_1,x_{M+1}]$, and   $\delta(\C')\le c$, then there exist real numbers $\tilde{w}_1,\cdots, \tilde{w}_M$ such that with $n=\a_1\delta(\C')^{-1}$,
\be\label{singquadrature}
\int_\RR P(x)dx=\sum_{k=1}^M \tilde{w}_kP(x_k), \qquad P\in \Pi_{n\sqrt{2}},
\ee
and
\be\label{singquadwtbd}
|\tilde{w}_k|\le c|x_{k+1}-x_k|, \qquad k=1,\cdots,M.
\ee
\end{theorem}

The proof of Theorem~\ref{singquadtheo} follows the now standard methods (e.g., \cite{mnw1, bitrep, jaenpap, modlpmz}). We first use the Markov--Bernstein inequality \eref{bernineq} with $p=1$ to prove 
the so called Marcinkiewicz--Zygmund inequalities (Lemma~\ref{mzlemma} below), and then use the Hahn--Banach theorem.

Before starting this program, we recall some finite--infinite range inequalities.

\begin{prop}\label{spuriousprop}
Let $n>0$, $1\le p, r\le\infty$, $P\in \Pi_n$. Then
\be\label{rangeineq}
\|P\|_{p,\RR\setminus [-2n,2n]}\le c\exp(-c_1n)\|P\|_{r, [-2n,2n]}.
\ee
Moreover, there exists $D>0$ such that with $B_n=(n\sqrt{2})(1+Dn^{-4/3})$, we have for $n\ge c$,
\be\label{rangeineql1}
\int_{\RR\setminus [-B_n, B_n]}|P(x)|dx \le (1/8)\int_{-B_n}^{B_n} |P(x)|dx.
\ee
\end{prop}

\begin{proof}\ 
 The estimate \eref{rangeineq} is proved in \cite[Proposition~6.2.8]{mhasbk} (and its proof). 
The estimate \eref{rangeineql1} is proved in \cite[Corollary~2.1]{bitrep}. (To reconcile the notation in \cite{bitrep}, we use $\a=2$ and $2n^2$ in place of $n$ which yields the interval denoted there by $\Delta_{n,\a}$ to be of the form $[-B_n, B_n]$  with a suitable value of $D$.) 
 \qed
\end{proof}

\begin{lemma}\label{mzlemma}
We assume the set up in Theorem~\ref{singquadtheo}. Then 
\be\label{pre_mzineq}
(3/4)\int_\RR |P(x)|dx \le \sum_{k=1}^M (x_{k+1}-x_k)|P(x_k)|\le (5/4)\int_\RR |P(x)|dx, \qquad P\in\Pi_{n\sqrt{2}}.
\ee
\end{lemma}

\begin{proof}\ 
 Let $P\in \Pi_{n\sqrt{2}}$, and $C=2^{-2/3}D$, where $D$ is defined in Proposition~\ref{spuriousprop}. Since  $[-A'_n,A'_n]\subseteq [x_1,x_{M+1}]$,   we obtain from \eref{rangeineql1} that for $n\ge c$ 
\be\label{pf3eqn1}
\int_{\RR\setminus [x_1, x_{M+1}]}|P(x)|dx \le (1/8)\int_{x_1}^{x_{M+1}}|P(x)|dx.
\ee
 For $k=1,\cdots,M$, we have
\begin{eqnarray*}
\lefteqn{\left|\int_{x_k}^{x_{k+1}}|P(x)|dx-(x_{k+1}-x_k)|P(x_k)|\right| \le \int_{x_k}^{x_{k+1}}\left||P(x)|-|P(x_k)|\right|dx}\\
&\le& \int_{x_k}^{x_{k+1}}|P(x)-P(x_k)|dx\le \int_{x_k}^{x_{k+1}}\int_{x_k}^{y}|P'(u)|dudx\\
&\le& (x_{k+1}-x_k)\int_{x_k}^{x_{k+1}}|P'(u)|du.
\end{eqnarray*}
Consequently, we deduce from \eref{pf3eqn1} and \eref{bernineq} that
\begin{eqnarray*}
\lefteqn{\left|\int_\RR |P(x)|dx -\sum_{k=1}^M (x_{k+1}-x_k)|P(x_k)|\right|}\\
& \le& \int_{\RR\setminus [x_1, x_{M+1}]}|P(x)|dx + \left|\int_{x_1}^{x_{M+1}}|P(x)|dx-\sum_{k=1}^M (x_{k+1}-x_k)|P(x_k)|\right|\\
&\le&(1/8)\int_\RR |P(x)|dx +\sum_{k=1}^M \left|\int_{x_k}^{x_{k+1}}|P(x)|dx-(x_{k+1}-x_k)|P(x_k)|\right|\\
&\le& (1/8)\int_\RR |P(x)|dx +\sum_{k=1}^M (x_{k+1}-x_k)\int_{x_k}^{x_{k+1}}|P'(u)|du\\
&\le& (1/8)\int_\RR |P(x)|dx +c\delta(\C')\int_\RR |P'(u)|du\\
&\le&   (1/8)\int_\RR |P(x)|dx +cn\delta(\C')\int_\RR |P(x)|dx.
\end{eqnarray*}
Therefore, choosing $\a_1$ sufficiently small, we obtain for $n=\a_1\delta(\C')^{-1}$,
$$
\left|\int_\RR |P(x)|dx -\sum_{k=1}^M (x_{k+1}-x_k)|P(x_k)|\right| \le (1/4)
\int_\RR |P(x)|dx.
$$
This completes the proof. \qed
\end{proof}

We are now in a position to complete the proof of Theorem~\ref{singquadtheo}.\\

\noindent
\textit{Proof of Theorem~\ref{singquadtheo}.} In this proof only, we define a norm on $\RR^M$ by
$$
\tn (z_1,\cdots,z_M)\tn =\sum_{k=1}^M (x_{k+1}-x_k)|z_k|,
$$
 the sampling operator $\mathcal{U} :\Pi_{n\sqrt{2}}\to \RR^M$ by $\mathcal{U}P=(P(x_1),\cdots,P(x_M))$, and denote the range of $\mathcal{U}$ by $V$. Then \eref{pre_mzineq} shows that the operator $\mathcal{U}$ is invertible on $V$, and we may define a linear functional on $V$ by
$$
x^*(\mathcal{U}P)=\int_\RR P(x)dx.
$$
The dual norm of this functional can be estimated easily using \eref{pre_mzineq}: 
$$
|x^*(\mathcal{U}P)|\le \int_\RR |P(x)|dx \le (5/4)\tn\mathcal{U}P\tn,
$$
so that the norm is $\le 5/4$. In view of the Hahn--Banach theorem, this functional can be extended from $V$ to $\RR^M$, where the extended functional has the same norm as $x^*$; i.e., $\le 5/4$. This extended functional can be identified with $(\tilde{w}_1,\cdots,\tilde{w}_M)\in\RR^M$. Then for $P\in\Pi_{n\sqrt{2}}$,
$$
\sum_{k=1}^M \tilde{w}_kP(x_k)=x^*(\mathcal{U}P)=\int_\RR P(x)dx,
$$
proving \eref{singquadrature}. The norm of the extended functional is 
$$
\max_{1\le k\le M}\frac{|\tilde{w}_k|}{x_{k+1}-x_k} \le (5/4).
$$
This proves \eref{singquadwtbd}.  \qed

Having proved Theorem~\ref{singquadtheo}, the proof of Theorem~\ref{quadtheo} is only a change of variables.

\noindent\textit{Proof of Theorem~\ref{quadtheo}.} 
Let $x_k=y_k\sqrt{2}$, $k=1,\cdots,M+1$, and $\C'=\{x_1,\cdots,x_{M+1}\}$. Then with $A_n'$, $\a_1$ as defined in Theorem~\ref{singquadtheo}, $\delta(\C')=\sqrt{2}\delta(\C)$,  and
 $[-A_n',A_n']\supset [x_1,x_{M+1}]$. Further, with $\a=\a_1/\sqrt{2}$, $n=\a\delta(\C)^{-1}=\a_1\delta(\C')^{-1}$. Therefore, Theorem~\ref{singquadtheo} yields $\tilde{w}_k$ satisfying \eref{singquadrature} and \eref{singquadwtbd}.
 
 If $P(y)=R_1(y)\exp(-y^2/2)$,
$Q(y)=R_2(y)\exp(-y^2/2)$, $R_1, R_2\in\PP_{n^2}$, then $x\mapsto R_1(x/\sqrt{2})R_2(x/\sqrt{2})\exp(-x^2/2)\in \Pi_{n\sqrt{2}}$.  Hence, with $w_k=\tilde{w}_k/\sqrt{2}$, \eref{singquadrature} implies that
\begin{eqnarray*}
\int_\RR P(y)Q(y)dy&=& \int_\RR R_1(y)R_2(y)\exp(-y^2)dy\\
&=&\frac{1}{\sqrt{2}}\int_\RR R_1(x/\sqrt{2})R_2(x/\sqrt{2})\exp(-x^2/2)dx\\
&=& \sum_{k=1}^M w_k R_1(y_k)R_2(y_k)\exp(-y_k^2)\\
&=&\sum_{k=1}^M w_kP(y_k)Q(y_k),
\end{eqnarray*}
which is \eref{quadrature}. Also, \eref{singquadwtbd} implies that
$$
|w_k|=\frac{1}{\sqrt{2}}|\tilde{w}_k| \le \frac{c}{\sqrt{2}}|x_{k+1}-x_k|=c|y_{k+1}-y_k|,
$$
which is \eref{quadwtbd}. \qed

\bhag{Wavelet--like representation}\label{mainsect}

We recall Definition~\ref{besovspacedef} of Besov spaces $B_{p,\rho,\gamma}$. Our first theorem is a characterization of these spaces in terms of an expansion of a function in $L^p$ based either on the Fourier--Hermite coefficients or values of the target function at arbitrary points on $\RR$.

Let $\aleph=\{\nu_n\}$ be a sequence of measures. We define the \textbf{frame operators} by
\be\label{tauopdef}
\tau_n(\aleph;f)=\left\{\begin{array}{ll}
\sigma_1(\nu_0;f), &\mbox{if  $n=0$,}\\
\sigma_{2^n}(\nu_n;f)-\sigma_{2^{n-1}}(\nu_{n-1};f), &\mbox{ if $n=1,2,\cdots$,}
\end{array}\right.
\ee
for all $f$ for which the operators involved are well defined. If each of the measures $\nu_n$ is the Lebesgue measure, we will omit the mention of the sequence in the notations. In this case, the operators are defined for $f\in L^1+L^\infty$. If each $\nu_n$ is a finitely supported measure, then the operators are defined for $f\in X^\infty$.

The following theorem is easy to deduce from  Theorem~\ref{goodapproxtheo} and \cite[Theorem~3.1]{chuiinterp}.

\begin{theorem}\label{globaltheo}
Let $1\le p\le\infty$, $\aleph=\{\nu_n\}$ be a sequence of measures such that each $\nu_n\in MZ(2^{n+1})$. Let $f\in X^p$. \\
{\rm (a)} We have
\be\label{lpseries}
f=\sum_{n=0}^\infty \tau_n(\aleph;f).
\ee
{\rm (b)} If $0<\rho\le \infty$, $0<\gamma<\infty$, then $f\in B_{p,\rho,\gamma}$ if and only if $\{\|f-\sigma_{2^n}(f)\|_p\}\in \seqb_{\rho,\gamma}$. In turn, $f\in B_{p,\rho,\gamma}$ if and only if $\{\|\tau_n(f)\|_p\}_{n=0}^\infty\in \seqb_{\rho,\gamma}$.\\
{\rm (c)}  Let $\aleph=\{\nu_n\}$ be a sequence of measures such that each $\nu_n\in MZ(2^{n+1})$,  $f\in X^\infty$, $0<\rho\le \infty$, and $0<\gamma<\infty$. Then $f\in B_{\infty,\rho,\gamma}$ if and only if $\{\|f-\sigma_{2^n}(\nu_n;f)\|_\infty\}\in \seqb_{\rho,\gamma}$. In turn, $f\in B_{\infty,\rho,\gamma}$ if and only if $\{\|\tau_n(\aleph; f)\|_\infty\}_{n=0}^\infty\in \seqb_{\rho,\gamma}$.\\
{\rm (d)} If $f\in L^2$ then
\be\label{framebd}
\|f\|_2^2\sim \sum_{n=0}^\infty \|\tau_n(f)\|_2^2.
\ee
\end{theorem}

The main purpose of this section is to show that \eref{lpseries} is a wavelet--like representation; i.e., the local behavior of the sequence $\{\tau_n(f)\}_{n=0}^\infty$  characterizes the membership of $f$ in local Besov spaces, defined below.

\begin{definition}\label{locbesovspacedef}
 If $x_0\in\RR$, the \textbf{local Besov space} $B_{p,\rho,\gamma}(x_0)$ is the space of all $f\in X^p$ with the following property : There exists a $\delta>0$ such that for every infinitely differentiable function $\phi$ supported on $\BB(x_0,\delta)$, $\phi f\in B_{p,\rho,\gamma}$.  \qed
\end{definition}

The wavelet--like representation property is described in the following theorem.
\begin{theorem}\label{localtheo}
Let $1\le p\le\infty$, $f\in X^p$,  $x_0\in\RR$, $0<\rho\le \infty$, and $0<\gamma<\infty$. The following statements are equivalent.\\
{\rm (a)} $f\in B_{p,\rho,\gamma}(x_0)$.\\
{\rm (b)} There exists a $\delta>0$ such that  $\{\|f-\sigma_{2^n}(f)\|_{p, \BB(x_0,\delta)}\}_{n=0}^\infty \in \seqb_{\rho,\gamma}$.\\
{\rm (c)} There exists a $\delta>0$ such that $\{\|\tau_n(f)\|_{p, \BB(x_0,\delta)}\}_{n=0}^\infty \in \seqb_{\rho,\gamma}$.
\end{theorem}

In the case of functions in $X^\infty$, one can obtain a similar theorem also based on the samples of the target function at arbitrary points.

\begin{theorem}\label{disclocaltheo}
Let  $f\in X^\infty$,  $x_0\in\RR$, $0<\rho\le \infty$, and $0<\gamma<\infty$. Let $\aleph=\{\nu_n\}$ be a sequence of measures such that each $\nu_n\in MZ(2^{n+1})$. The following statements are equivalent.\\
{\rm (a)} $f\in B_{\infty,\rho,\gamma}(x_0)$.\\
{\rm (b)} There exists a $\delta>0$ such that $\{\|f-\sigma_{2^n}(\nu_n;f)\|_{\infty, \BB(x_0,\delta)}\}_{n=0}^\infty \in \seqb_{\rho,\gamma}$.\\
{\rm (c)} There exists a $\delta>0$ such that $\{\|\tau_n(\aleph;f)\|_{\infty, \BB(x_0,\delta)}\}_{n=0}^\infty \in \seqb_{\rho,\gamma}$.
\end{theorem}

We will prove Theorem~\ref{disclocaltheo} in some detail, and then indicate the changes required to prove Theorem~\ref{localtheo}.

\noindent\textit{Proof of Theorem~\ref{disclocaltheo}.}
In this proof, we will choose and fix an integer $S>\gamma+3$. All constants may depend upon $x_0$, $\delta$, and S. 

Let (a) hold, and $\delta>0$ be such that for every infinitely differentiable function $\phi$ supported on $\BB(x_0,\delta)$, $\{E_{2^n,\infty}(\phi f)\}_{n=0}^\infty \in \seqb_{\rho,\gamma}$. In this part of the proof, let  $\phi$ be an infinitely differentiable function suppored on $\BB(x_0,\delta)$ and equal to $1$ on $\BB(x_0,3\delta/4)$. All the constants in this proof will depend upon $x_0$ and $\delta$. We  use 
the first estimate in \eref{hermite_lockernest} and \eref{fundaineq} (with $S+1$ in place of $S$)   to conclude that  for $x\in I=\BB(x_0,\delta/2)$, 
\bea\label{pf6eqn1}
\lefteqn{|\sigma_{2^n}(\nu_n;(1-\phi)f,x)|=\left|\int_{\RR\setminus \BB(x_0,3\delta/4)}(1-\phi(y))f(y)\Phi_n(x,y)d\nu_n(y)\right|}\nonumber\\
&\le& c\|f\|_\infty \int_{\RR\setminus \BB(x_0,3\delta/4)}|\Phi_n(x,y)|d|\nu_n|(y)\nonumber\\
&\le& c\|f\|_\infty \int_{\RR\setminus \BB(x,\delta/4)}|\Phi_n(x,y)|d|\nu_n|(y)
\le c2^{-nS}\|f\|_\infty.
\eea
Therefore, \eref{goodapprox} leads to
\bea\label{pf6eqn2}
\lefteqn{\|f-\sigma_{2^n}(\nu_n;f)\|_{\infty, I}=\|\phi f-\sigma_{2^n}(\nu_n;f)\|_{\infty, I}}\nonumber\\
&\le& \|\phi f-\sigma_{2^n}(\nu_n;\phi f)\|_{\infty, I}+\|\sigma_{2^n}((1-\phi)f\|_{\infty, I}\nonumber\\
&\le& c\left\{E_{2^{n-1},\infty}(\phi f) 
+2^{-nS}\|f\|_\infty\right\}.
\eea
Since  $S>\gamma+3$, each of the sequences $\{E_{2^{n-1},\infty}(\phi f)  \}_{n=0}^\infty$ and
 $\{2^{-nS}\|f\|_\infty  \}_{n=0}^\infty$  belongs to $\seqb_{\rho,\gamma}$. Therefore, \eref{pf6eqn2} implies the statement in part (b).

Conversely, let part (b) hold, and $\phi$ be any infinitely differentiable function supported on $I=\BB(x_0,\delta)$.  Since $\phi$ is in particular $2S$ times continuously differentiable,  the direct theorem of approximation \cite[Theorem~4.2.1]{mhasbk} shows that  for $n\ge c$, there exists $R_n\in\Pi_{2^n}$ such that $\|R_n\|_\infty\le c$, and
\be\label{pf6eqn3}
\|\phi-R_n\|_\infty \le c2^{-nS}.
\ee
Therefore, using the notation introduced in \eref{tildedegapproxdef},
\begin{eqnarray*}
\lefteqn{\tilde{E}_{2^{n+1},\infty}(\phi f)\le \|\phi f-R_n\sigma_{2^n}(\nu_n;f)\|_\infty}\\
&\le& \|\phi(f-\sigma_{2^n}(\nu_n;f))\|_\infty + \|(\phi-R_n)\sigma_{2^n}(\nu_n;f)\|_\infty \\
&\le & c\left\{ \|(f-\sigma_{2^n}(\nu_n;f)\|_{\infty,I} + \|\phi-R_n\|_\infty\|\sigma_{2^n}(\nu_n;f)\|_\infty\right\}\\
&\le& c\left\{\|(f-\sigma_{2^n}(\nu_n;f)\|_{\infty,I} + c2^{-nS}\|f\|_\infty \right\}.
\end{eqnarray*}
As before, the statement in part (b) now leads to $\{\tilde{E}_{2^n,\infty}(\phi f)\}_{n=0}^\infty \in \seqb_{\rho,\gamma}$. In view of Corollary~\ref{wtchangecor}, this implies the statement in part (a).

The equivalence of parts (b) and (c) follows from \eref{lpseries}, and an application of the discrete Hardy inequalities \cite[p.~27]{devlorbk}. \qed

\noindent\textit{Proof of Theorem~\ref{localtheo}.}
The proof is almost verbatim the same as that of Theorem~\ref{disclocaltheo}, except for one difference, which we now point out. We continue the notation as in the proof of (a)$\Rightarrow$ (b). All the constants in this proof will depend upon $x_0$ and $\delta$. As shown in \eref{pf6eqn1} (with the Lebesgue measure in place of $\nu_n$),
\be\label{pf7eqn1}
\|\sigma_{2^n}((1-\phi)f\|_{\infty, I}\le c2^{-nS}\|f\|_\infty, \qquad f\in L^\infty.
\ee
If $f\in L^1$, then \eref{hermite_lockernest} (with $S+1$ in place of $S$) implies that
\bea\label{pf7eqn2}
\lefteqn{\int_I |\sigma_{2^n}((1-\phi)f,x)|dx}\nonumber\\
& \le& \int_I \int_{\RR\setminus \BB(x_0,3\delta/4)}|(1-\phi(y))f(y)||\Phi_n(x,y)|dydx\nonumber\\
&\le& c\|f\|_1\sup_{y\in \RR\setminus \BB(x_0,3\delta/4)}\int_I |\Phi_n(x,y)|dx \le c2^{-nS}\|f\|_1.
\eea
The Riesz--Thorin interpolation theorem applied with the operator\\ $f\mapsto \sigma_{2^n}((1-\phi)f)$, together with \eref{pf7eqn1} and \eref{pf7eqn2} now implies that for $1\le p\le\infty$,
$$
\|\sigma_{2^n}((1-\phi)f\|_{p, I}\le c2^{-nS}\|f\|_p, \qquad f\in L^p.
$$
The remainder of the proof is almost verbatim the same as that of Theorem~\ref{disclocaltheo}. \qed

\begin{acknowledgement}
The research of HNM is supported
in part by Grant W911NF-15-1-0385 from the U. S. Army Research Office. We thank  the editors for their  kind invitation to submit this paper.
\end{acknowledgement}
%

\bibliographystyle{plain}
\bibliography{hermitebib}
\end{document}